\documentclass[12pt]{article}
\usepackage[dvips]{graphicx}
\usepackage[margin=16pt,font=small,labelfont=bf]{caption}
\usepackage[dvips,colorlinks]{hyperref}
\usepackage{amsmath,amssymb,amsthm,array,natbib,url,fullpage}
\graphicspath{{./}{eps/}}

\newcommand{\thd}{\mathchoice{{\textstyle\frac13}} {{\textstyle\frac13}}
                {{\scriptscriptstyle\frac13}}{{\scriptscriptstyle\frac13}}}

\ifx\newcolumntype\undefined
\else
  \newcolumntype{C}{>{$}c<{$}}
  \newcolumntype{L}{>{$}l<{$}}
  \newcolumntype{R}{>{$}r<{$}}
\fi
\newcommand{\omr}{{\bar\rho}}
\newcommand{\AR}[1]{\textsf{AR}(#1)}
\newcommand{\OU}{\textsf{OU}}

\newcommand{\GA}[1]{\Ga\big(#1\big)}
\newcommand{\gA}{\mathfrak{A}}
\newcommand{\cN}{\mathcal{N}}

\newcommand{\pat}{\frac{\pi\alpha}{2}}
\newcommand{\Del}[1]{{\ensuremath{\Delta_{#1}}}}
\newcommand{\BP}{\textsf{BP}}
\newcommand{\SaS}{\textsf{S$\alpha$S}}
\newcommand{\Ga}{\mathsf{Ga}}
\newcommand{\cT}{{\mathcal{T}}}

\newcommand{\bbC}{{\mathbb{C}}}
\newcommand{\bbN}{{\mathbb{N}}}
\newcommand{\bbR}{{\mathbb{R}}}
\newcommand{\bbZ}{{\mathbb{Z}}}

\newcommand{\etc}{\emph{etc}}
\newcommand{\ie}{\emph{i.e.{}}}
\newcommand{\one}{\mathbf{1}}
\newcommand{\bone}[1]{\one_{\{#1\}}}
\newcommand{\Sec}[1]{Section\thinspace(\ref{#1})}

\newcommand{\Eqn}[1]{Eqn\thinspace(\ref{#1})}
\newcommand{\bet}[1]{\left [#1\right ]} 
\newcommand{\cet}[1]{\left (#1\right )} 
\newcommand{\set}[1]{\left\{#1\right\}} 
\newcommand{\ind}{\mathrel{\mathop{\sim}\limits^{\mathrm{ind}}}}
\newcommand{\iid}{\mathrel{\mathop{\sim}\limits^{\mathrm{iid}}}}
\renewcommand{\P}{{\mathsf{P}}} \newcommand{\E}{{\mathsf{E}}}
\newcommand{\Po}{\mathsf{Po}}
\newcommand{\Corr}{{\mathsf{Corr}}}
\def\need#1{\vskip0pt plus#1in\penalty-250\vskip0pt plus-#1in}%
\newcommand{\Be}{\mathsf{Be}}
\newcommand{\cB}{{\mathcal{B}} }
\newcommand{\cG}{{\mathcal{G}}}

\newcommand{\Fig}[1]{Figure\thinspace(\ref{#1})}

\newcommand{\V}{{\mathsf{V}}}\newcommand{\Cov}{{\mathsf{Cov}}}
\newcommand{\NB}{\mathsf{NB}}\newcommand{\No}{\mathsf{No}}
\newcommand{\eps}{\epsilon}\newcommand{\hide}[1]{}
\newcommand{\pg}{\emph{p.}\thinspace}

\newcommand{\half}{{\mathchoice{{\textstyle\frac12}} {{\textstyle\frac12}}
                {{\scriptscriptstyle\frac12}}{{\scriptscriptstyle\frac12}}}}

\newtheorem{prop}{Proposition}
\newtheorem{thm}{Theorem}

\newcommand{\Ca}{\mathsf{Ca}}\newcommand{\St}{\mathsf{St}}
\newcommand{\sgn}{\mathop{\mathrm{sgn}}}

\def\Strut{\vrule width0pt height 16pt depth 4pt}%
\def\dat$Dat#1: #2 #3(#4) ${#2}
\newif{\ifBib} \Bibfalse 
\begin{document}
\thispagestyle{empty}

\title {Lecture Notes on Stationary Gamma Processes}
\author{Robert L. Wolpert
}
\date{\today}
\maketitle
\begin{abstract}\noindent
  For each $\lambda>0$ and every square-integrable infinitely-divisible
  (ID) distribution there exists at least one stationary stochastic process
  $t\mapsto X_t$ with the specified distribution for $X_1$ and with
  first-order autoregressive (\AR1) structure in the sense that the
  autocorrelation of $X_s$ and $X_t$ is $\exp(-\lambda|s-t|)$ for all
  indices $s,t$.  For the special case of the standard Normal distribution,
  the process $X_t$ is unique--- namely, the first-order autoregressive
  Ornstein-Uhlenbeck velocity process.  The process $X_t$ is also uniquely
  determined if $X_1$ is accorded the unit rate Poisson distribution.

  For the Gamma distribution, however, $X_t$ is \emph{not} determined
  uniquely. In these lecture notes we describe six distinct processes with
  the same univariate marginal distributions and \AR1 autocorrelation
  function.  We explore a few of their properties and describe methods of
  simulating their sample paths.
\end{abstract}
\section{Introduction}\label{s:intro}

For fixed $\alpha,\beta>0$ these notes present six different stationary
time series, each with Gamma $X_t\sim \Ga(\alpha,\beta)$ univariate
marginal distributions and autocorrelation function $\rho^{|s-t|}$ for
$X_s,~X_t$.  Each will be defined on some time index set $\cT$, either
$\cT=\bbZ$ or $\cT=\bbR$.

Five of the six constructions can be applied to other Infinitely Divisible
(ID) distributions as well, both continuous ones (normal, $\alpha$-stable,
\etc.) and discrete (Poisson, negative binomial, \etc.).  For specifically
the Poisson and Gaussian distributions, all but one of them (the Markov
change-point construction) coincide--- essentially, there is just one
``\AR1-like'' Gaussian process (namely, the \AR1 process in discrete time,
or the Ornstein-Uhlenbeck process in continuous time), and there is just one
\AR1-like Poisson process.  For other ID distributions, however, and in
particular for the Gamma, each of these constructions yields a process with
the same univariate marginal distributions and the same autocorrelation but
with different joint distributions at three or more times.

First, by ``$\Ga(\alpha,\beta)$'' we mean the Gamma distribution with mean
$\alpha/\beta$ and variance $\alpha/\beta^2$, \ie, with shape parameter
$\alpha$ and \emph{rate} parameter $\beta$.  The pdf and chf are given by:
\begin{align}
 f(x\mid \alpha,\beta) 
    &= \frac{\beta^\alpha}{\Gamma(\alpha)}~
       x^{\alpha-1}e^{-\beta x}\bone{x>0}\notag\\
  \chi(\omega\mid \alpha,\beta)
    & = (1-i\omega/\beta)^{-\alpha}\notag\\
    & = \exp\set{-\int_{\bbR_+} e^{i\omega u}~
        \alpha e^{-\beta u}\,u^{-1}\,du}.\label{e:levy}
\end{align}
The sum $\xi_+:=\sum \xi_j$ of independent random variables $\xi_j\ind
\Ga(\alpha_j,\beta)$ with the same rate parameter also has a Gamma
distribution, $\xi_+\sim\Ga(\alpha_+,\beta)$, if $\set{\alpha_j}
\subset\bbR_+$ are summable with sum $\alpha_+ :=\Sigma\alpha_j <\infty$.
\Eqn{e:levy} shows that the ``L\'evy measure'' for this distribution is
$\nu(du) = \alpha e^{-\beta u}\, u^{-1}\bone{u>0}du$.

From the Poisson representation of ID distributions
\[ X = \int_\bbR u \cN(du) \]
for $\cN(du)\sim\Po\big(\nu(du)\big)$ we can show that the extremal
properties of ID random variables depend only on the L\'evy measure
$\nu(du)$: for large $u\in\bbR_+$,
\begin{align}
\P[ X > u ] &\approx \P\big[\cN\big((u,\infty)\big)\ne\emptyset]\notag\\
            &= 1-\exp\big(-\nu\big((u,\infty)\big)\notag\\
            &\approx \nu\big((u,\infty)\big)
             \approx (\alpha/\beta u)e^{-\beta u},\label{e:aSx}
\end{align}
from which one might mount a study of the multivariate extreme properties
of the six processes presented below.

\section{Six Stationary Gamma Processes}\label{s:6procs}
\subsection{The Gamma $\AR1$ Process}\label{ss:ar1}

Fix $\alpha,\beta>0$ and $0\le\rho<1$.  Let $X_0\sim\Ga(\alpha,\beta)$ and  
for $t\in\bbN$ define $X_t$ recursively by 
\begin{align} X_t &:= \rho X_{t-1} + \zeta_t \label{e:ar1}
\end{align}
for iid $\set{\zeta_t}$ with chf
$\E e^{i\omega\zeta_t} =(1-i\omega/\beta)^{-\alpha} (1-i\rho
\omega/\beta)^\alpha =\big[\frac{\beta-i\omega}
{\beta-i\rho\omega}\big]^{-\alpha}$ (easily seen to be positive-definite,
with L\'evy measure
$\nu_\zeta (du)=\alpha\bet{e^{-\beta u}-e^{-\beta u/\rho}}
u^{-1}\bone{u>0}du$).  A simple way of generating $\set{\zeta_t}$ with this
distribution from \citep{Walk:2000} is presented in the Appendix.  The
process $\{X_t\}$ has Gamma univariate marginal distribution
$X_t\sim\Ga(\alpha,\beta)$ for every $t\in\bbR_+$ and, at consecutive times
$0,1$, joint chf
\begin{align}
  \chi(s,t) &= \E\exp(i s X_0 + i t X_1)\notag\\
  &= \E\exp(i (s+\rho t) X_0 + i t \zeta_1)\notag\\
  &= \Big[\frac{(1-i(s+\rho t)/\beta)~(1-it/\beta)}
               {1-it\rho/\beta}\Big]^{-\alpha}. \label{e:ar1-chf}
\end{align}
Since this is asymmetric in $s,t$, the process is not time-reversible; this
is also evident from the observation that 
\[\P[X_t\ge\rho X_{t-1}] = \P[\zeta_t\ge0] = 1 >
 \P[\zeta_t\le X_{t-1}(1-\rho^2)/\rho] = \P[X_{t-1}\ge\rho X_t].
\]
This process has marginal distribution $X_t\sim \Ga(\alpha,\beta)$ at all
times $t$, and has autocorrelation $\Corr\big(X_s, X_t) =\rho^{|s-t|}$
(easily found from either \eqref{e:ar1} or \eqref{e:ar1-chf}).  It is
clearly Markov (from \eqref{e:ar1}), and can be shown to have
infinitely-divisible (ID) multivariate marginal distributions of all
orders, since the $\{\zeta_t\}$ are ID.

\need2
\subsection{Thinned Gamma Process}\label{ss:thin}

If $X\sim\Ga(\alpha_1,\beta)$ and $Y\sim\Ga(\alpha_2,\beta)$ are
independent, then $Z:=X+Y\sim\Ga(\alpha_+,\beta)$ and $U:=X/Z\sim
\Be(\alpha_1,\alpha_2)$ are also independent (where $\alpha_+:=\alpha_1
+\alpha_2$), because under the change of variables $x=u z$, $y=(1-u)z$
with Jacobian $J(u,z)=z$,
\begin{align*}
  f(u,z) 
&= \set{\frac{\beta^{\alpha_1}}
       {\Gamma(\alpha_1)} x^{\alpha_1-1} e^{-\beta x}}
   \set{\frac{\beta^{\alpha_2}}
       {\Gamma(\alpha_2)} y^{\alpha_2-1} e^{-\beta y}}\,J(u,z)\\
&= \set{\frac{\beta^{\alpha_+}}{\Gamma(\alpha_1)\Gamma(\alpha_2)}}
   u^{\alpha_1-1}(1-u)^{\alpha_2-1} z^{\alpha_+-2} e^{-\beta z}z \\
&= \set{\frac{\Gamma(\alpha_+)}{\Gamma(\alpha_1)\Gamma(\alpha_2)}
   u^{\alpha_1-1}(1-u)^{\alpha_2-1}}\quad
   \set{\frac{\beta^{\alpha_+}}{\Gamma(\alpha_+)} z^{\alpha_+-1} e^{-\beta z}}.
\end{align*}
Thus if $Z\sim\Ga(\alpha_1+\alpha_2,\beta)$ and
$U\sim\Be(\alpha_1,\alpha_2)$ are independent then
$X=UZ\sim\Ga(\alpha_1,\beta)$ and $Y=(1-U)Z\sim\Ga(\alpha_2,\beta)$ are
independent too.
Let
\begin{align*}
X_0&\sim\Ga(\alpha,\beta)\\
\intertext{and, for $t\in\bbN$ (or $t\in-\bbN$, resp.) set} 
X_{t} &:= \xi_t + \zeta_t
\intertext{where}
\xi_t &:=B_t\cdot X_{t-1},\qquad B_t\sim\Be(\alpha\rho, \alpha\omr),\\
\zeta_t&\sim\Ga(\alpha\omr,\beta)
\intertext{(or $\xi_t =B_t\cdot X_{t+1}$, resp.), where $\omr:=(1{-}\rho)$
  and all the $\set{B_t}$ and $\set{\zeta_t}$ are independent.}
\end{align*}
Then, by induction, $\xi_t\sim\Ga(\alpha\rho,\beta)$ and $\zeta_t\sim
\Ga(\alpha\omr,\beta)$ are independent, with sum $X_t\sim\Ga(\alpha,
\beta)$.  Thus $\{X_t\}$ is a Markov process with Gamma univariate marginal
distribution $X_t\sim\Ga(\alpha, \beta)$, now with symmetric joint chf
\begin{align}
         \chi(s,t) &= \E\exp(i s X_0 + i t X_1)\notag\\
  &= \E\exp\set{i s (X_0-\xi_1) + i (s+t)\xi_1 + i t \zeta_1}\notag\\
  &=(1-is/\beta)^{-\alpha\omr}
    (1-i(s+t)/\beta)^{-\alpha\rho}
    (1-it/\beta)^{-\alpha\omr}
    \label{e:thin-chf}
\end{align}
and autocorrelation $\Corr\big(X_s, X_t)=\rho^{|s-t|}$.  The process of
passing from $X_{t-1}$ to $\xi_t=X_{t-1}B_t$ is called \emph{thinning}, so
$X_t$ is called the \emph{thinned gamma process}.  A similar construction
is available for any ID marginal distribution.

\subsection{Random Measure Gamma Process}\label{ss:rm}

Let $\cG(dx\,dy)$ be a countably additive random measure that assigns
independent random variables $\cG(A_i)\sim\Ga(\alpha|A_i|,\beta)$ to
disjoint Borel sets $A_i\in\cB (\bbR^2)$ of finite area $|A_i|$ (this is
possible by the Kolmogorov consistency conditions, and is illustrated in
the Appendix) and, for $\lambda:=-\log\rho$, consider the collection of
sets:
\[ G_t :=\set{(x,y):~x\in\bbR, ~0 \le y < \lambda e^{-2\lambda |t-x|}}\]
\begin{figure}[!htp]
\begin{center}
\includegraphics[height=70mm]{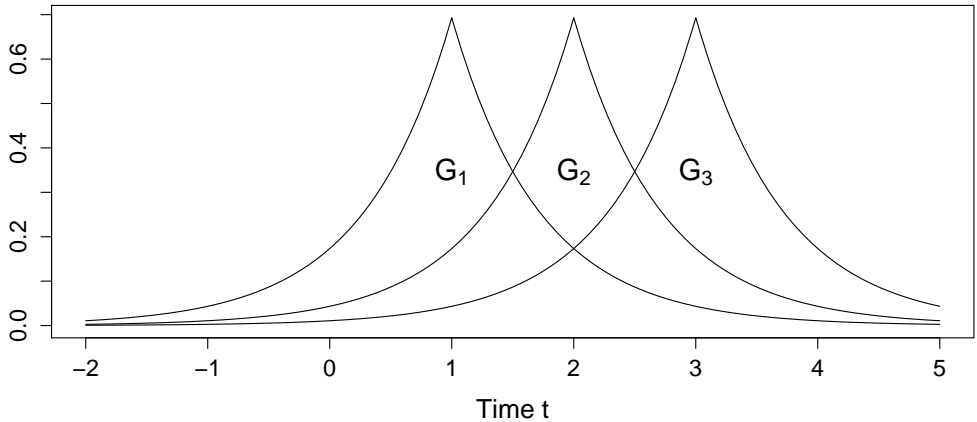}
\vspace*{-5mm}
\caption{\label{f:Gt}Random measure construction of process $X_t=\cG(G_t)$}
\end{center}
\vspace*{-5mm}
\end{figure}%
shown in \Fig{f:Gt} whose intersections have area $|G_s\cap G_t| =
e^{-\lambda|s-t|}$.  For $t\in\cT=\bbR$, set \iffalse
\begin{subequations}\label{e:meas}
\begin{align}
 X_t &:= \cG(G_t)\label{e:meas.X}
\intertext{for the set}
 G_t &:=\set{(x,y):~x\in\bbR, ~0 \le y < \lambda~e^{-2\lambda
    |t-x|}}\label{e:meas.G}
\end{align}\end{subequations}
\else
\begin{equation}\label{e:meas}
 X_t := \cG(G_t).
\end{equation}
\fi 
For any $n$ times $t_1<t_2<\dots<t_n$ the sets $\set {G_{t_i}}$ partition
$\bbR^2$ into $n(n+1)/2$ sets of finite area (and one with infinite area,
$(\cup G_{t_i})^c$), so each $X_{t_i}$ can be written as the sum of some
subset of $n(n+1)/2$ independent Gamma random variables.  In particular,
any $n=2$ variables $X_s$ and $X_t$ can be written as
\[ X_s=\cG(G_s\backslash G_t)+\cG(G_s \cap G_t),\qquad
   X_t=\cG(G_t\backslash G_s)+\cG(G_s \cap G_t)
\]
just as in the thinning approach of \Sec{ss:thin}, so both 1-dimensional and
2-dimensional marginal distributions for the random measure process coincide
with those for the thinning process.  Again the joint chf is
\begin{align}
         \chi(s,t) &= \E\exp(i s X_0 + i t X_1)\notag\\
  &=(1-is/\beta)^{-\alpha\omr}
    (1-i(s+t)/\beta)^{-\alpha\rho}
    (1-it/\beta)^{-\alpha\omr}
    \label{e:rm-chf}
\end{align}
and the autocorrelation is $\Corr\big(X_s, X_t)= \exp \big(-\lambda
|s-t|\big)$ or, for integer times, $\rho^{|s-t|}$ for $\rho
:=\exp(-\lambda)$.  The distribution for consecutive \emph{triplets}
differs from those of the Thinned Gamma Process, however, an illustration
that the thinning process is Markov but the random measure is not.  The
Random Measure process does feature infinitely-divisible (ID) marginal
distributions of all orders, while the thinned process does not.

\subsection{The Markov change-point Gamma Process}\label{ss:mark-cp}
Let $\{\zeta_n:~n\in\bbZ\}\iid\Ga(\alpha,\beta)$ be iid Gamma random
variables and let $N_t$ be a standard Poisson process indexed by $t\in\bbR$
(so $N_0=0$ and $(N_t-N_s)\sim\Po(t-s)$ for all $-\infty<s<t<\infty$, with
independent increments), and set
\[ X_t := \zeta_n,\qquad n={N_{\lambda t}}.\]
Then each $X_t\sim\Ga(\alpha,\beta)$ and, for $s,t\in\bbR$, $X_s$ and $X_t$
are either identical (with probability $\rho^{|s-t|}$) or independent---
reminiscent of a Metropolis MCMC chain.  The chf is
\begin{align}
 \chi(s,t) &= \E\exp(i s X_0 + i t X_1)\notag\\
           &= \rho \big(1-i(s+t)/\beta\big)^{-\alpha}
            + \omr (1-is/\beta)^{-\alpha}(1-it/\beta)^{-\alpha}\label{e:met-chf}
\end{align}
and once again the marginal distribution is $X_t\sim\Ga(\alpha,\beta)$ and
the autocorrelation function is $\Corr\big(X_s, X_t)=\rho^{|s-t|}$.

\subsection{The Squared O-U Gamma Diffusion}\label{ss:OU}
Let $\{Z_i\}\iid\OU(\lambda/2,1)$ be independent Ornstein-Uhlenbeck
velocity processes, mean-zero Gaussian processes with covariance
$\Cov\big(Z_i(s), Z_j(t)\big) =\exp\big(-{\frac\lambda2}|s-t|\big)
\delta_{ij}$, and set
\[ X_t := \frac{1}{2\beta}\sum_{i=1}^n Z_i(t)^2 \]
for $n\in\bbN$ and $\beta\in\bbR_+$.  Note $Z_i(t)\sim\No(0,1)$, so $\E
Z_i(t)^2=1$ and $\E Z_i(t)^4=3$; it follows that $\E Z_i(s)^2
Z_i(t)^2=1+2\exp(-\lambda|s-t|)$.  Then $X_t\sim\Ga(\alpha,\beta)$ for
$\alpha=n/2$, with $\E X_s=\alpha/\beta$ and
\begin{align*}
  \E X_s X_t &= \frac{1}{4\beta^2}\set{n\big(1+2\exp(-\lambda|s-t|)\big)
  +n(n-1)\strut}\\
     &= \frac{\alpha^2+\alpha\exp\big(-\lambda|s-t|\big)}{\beta^2}
\end{align*}
so the autocovariance is $\Cov\big(X_s,X_t) =
\frac{\alpha}{\beta^2}~e^{-\lambda|s-t|}$ and the autocorrelation at
integer times is $\Corr\big(X_s, X_t)=\rho^{|s-t|}$ for $\rho :=\exp
(-\lambda)$.  The chf at consecutive integer times is
\begin{align}
         \chi(s,t) &= \E\exp(i s X_0 + i t X_1)\notag\\
                   &= \big(1-i(s+t)/\beta
                      -st(1-\rho)/\beta^2\big)^{-\alpha},\label{e:ou-chf}
\end{align}
distinct from \eqref{e:ar1-chf}, \eqref{e:thin-chf}=\eqref{e:rm-chf}, and
\eqref{e:met-chf}, so this process is new.  It\^o's formula is used in
\citep{Wolp:Brow:2017} to show that $X_t$ has stochastic differential
equation (SDE) representation
\begin{align}
   X_t &= X_0 - \int_0^t 2\lambda\big(X_s-\alpha/\beta\big)\,ds
           + 2\sqrt{\lambda/\beta} \int_0^t \sqrt{X_s}\,dW_s\label{e:sde}
\end{align}
and hence has generator $\gA\phi(x) = (\partial/\partial\eps)
\E[\phi(X_{t+\eps})\mid X_t=x]\big|_{\eps=0}$ given by
\begin{align}
\gA \phi(x) &= -2\lambda(x-\alpha/\beta)\phi'(x) + (2\lambda/\beta) x
                \phi''(x),\label{e:OU-gen}
\end{align}
which we will use to distinguish this process from that of \Sec{ss:cont}.
While the construction above required half-integer values for $\alpha$,
\eqref {e:ou-chf} is positive-definite and the SDE \eqref{e:sde} has a
unique strong solution for all $\alpha>0$, so a time-reversible stationary
Markov diffusion process exists with this distribution.  
\citet [Eqn (17)]{Cox:Inge:Ross:1985}, building on \citep{Fell:1951}, 
found the transition kernel for this process:
\begin{align*}
  f(y,t\mid x,s) &= c e^{-u-v} \Big(\frac v u\Big)^{(\alpha-1)/2}
                      I_{(\alpha-1)}\big(2\sqrt{u v}\big),
\end{align*}
where
\[
  c := \beta/\big[1-e^{-2\lambda|t-s|}\big],\qquad
  u := c y e^{-2\lambda|t-s|},\qquad
  v := c x.
\]
where $I_q(x)$ denotes the modified Bessel function of the first kind of
order $q$.  With this one can construct likelihood functions and generate
posterior samples, MLEs, \etc. for the parameters of this process.
\citeauthor{Cox:Inge:Ross:1985} show that the process $X_t$ is strictly
positive at all times if $\alpha\ge1$, but occasionally reaches zero for
$\alpha<1$.

\subsection{Continuously Thinned Gamma Process}\label{ss:cont}
Pick a large integer $n$ and set $\eps:=1/n$, $q:=\exp(-\lambda\eps)$, and
$p:=1-q=\lambda\eps+o(\eps)$.  Draw $X_0\sim\Ga(\alpha,\beta)$ and, for
integers $i,j\in\bbN$, draw independently
\[\zeta_i\sim\Ga(\alpha p, \beta) \qquad
      b_j\sim\Be(\alpha p, \alpha q). \]
Set:
\begin{alignat*}5
X_0&=X_0\\
X_\eps&=X_0 (1-b_1)&\,+\,&\zeta_1\\
X_{2\eps}&=X_\eps\, (1-b_2)&&&\,+\,&\zeta_2\\
         &=X_0 (1-b_1) (1-b_2) &\,+\,&\zeta_1 (1-b_2) &\,+\,& \zeta_2\\
X_{3\eps}&=X_0 (1-b_1) (1-b_2) (1-b_3) 
         &\,+\,&\zeta_1 (1-b_2) (1-b_3) &\,+\,&\zeta_2 (1-b_3) &\,+\,& \zeta_3\\
\end{alignat*}
and, in general,
\begin{subequations}
\begin{align}
  X_{k\eps}&=X_0 \prod_{j=1}^k (1-b_j) + \sum_{i=1}^k \left\{\zeta_i
                \prod_{i<j\le k} (1-b_j)\right\}.\label{e:Xt-Sum}
\intertext{In the limit as $n\to\infty$ and $k\eps\to t$ the products
  converge to the product integral of the beta process introduced by
  \citet[\S3] {Hjor:1980} \citep[and described lucidly by][\S2]
  {Thib:Jord:2007} and the sum to an ordinary gamma stochastic integral,}
   X_t &=X_0 \prod_{s\in(0,t]} [1-dB(s)] + 
             \int_0^t \left\{\prod_{s\in(r,t]} [1-dB(s)]\right\}\,\zeta(dr),
             \label{e:Xt-SI} 
\end{align}
\end{subequations}
where $\zeta(dr)\sim\Ga\big(\alpha\lambda\, dr,\beta\big)$ is a Gamma random
measure and $B(s)\sim\BP\big(\alpha,\lambda\,ds\big)$ is a Beta process, \ie,
an SII L{\'e}vy process with L{\'e}vy measure
\[ \nu_B(du)= \lambda\alpha\,u^{-1}(1-u)^{\alpha-1}\,\bone{0<u<1}\,du \]
with constant ``concentration function'' $\alpha(s)\equiv\alpha$ and
translation-invariant ``base measure'' $\lambda(ds)=\lambda ds$.  The
product integral can be written as a ratio
\begin{align*}
\prod_{s\in(r,t]} [1-dB(s)] &= \frac{1-F(t)}{1-F(r)}
\end{align*}
where $F(t)=\prod_{s\in(t,\infty)} [1-dB(s)]$ satisfies
\begin{align}
\frac{dF(t)}{1-F(t)} &= dB(t),\qquad
      B_t = \int_{(0,t]} \frac{dF(s)}{1-F(s)}. \label{e:dFdB}
\end{align}

\subsubsection{Generator}\label{s:gen}
The Gamma process of \Eqn{e:Xt-SI} is stationary and Markov, with generator
\begin{align}
\gA \phi(x) = &\int_0^\infty \big[\phi(x+u)-\phi(x)\big]\,\alpha\lambda
          u^{-1}e^{-\beta u}\,du\notag\\
          + &\int_0^x \big[\phi(x-u)-\phi(x)\big]\,\alpha\lambda
          u^{-1}(1-u/x)^{\alpha-1} \,du\label{e:cthin-gen}
\end{align}
Because this differs from \eqref{e:OU-gen} (and in particular because it is
a non-local operator, showing $X_t$ has jumps), this process is new.
Once again the one-dimensional marginal distributions are
$X_t\sim\Ga(\alpha, \beta)$ and the autocorrelation is $\Corr(X_s,X_t)
=\exp\big(-\lambda |s-t|\big)$.

\section{Discussion}\label{s:disc}
We have now constructed six distinct processes that share the same
univariate marginal distribution and autocorrelation function, but which
all differ in their $n$-variate marginal distributions for $n\ge3$.
Some are Markov, some not; some are time-reversible, some not; some have ID
marginal distributions of all orders, some don't.  Similar methods can be
used to construct $\AR1$-like processes with \emph{any} ID marginal
distribution, such as those listed on \pg \pageref{t:id}; only in the two
cases of Gaussian and Poisson do all these constructions coincide.  Many of
these are useful for modeling time-dependent phenomena whose dependence
falls off over time, but for which traditional Gaussian methods are
unsuitable because of heavy tails, or integer values, or positivity, or for
other reasons.  I know of very little work (yet!) exploring inference for
processes like these; a beginning appears in \citep[\S3]{Wang:2013}.  Code
in \texttt{R} to generate samples from each of these six processes is
available on request.

The author would like to thank Lawrence D. Brown for many fruitful
discussions about ID processes, and Victor Pe\~na for helping develop
methods of distinguishing these six processes from observations.

\vfill\newpage
\appendix
\section*{Appendix}\label{s:app}
\newcommand{\orr}{{\textstyle\frac{1-\rho}{\rho}}}
\begin{prop}[\citet{Walk:2000}]
The innovations $\zeta_t$ in \Eqn{e:ar1} can be constructed successively
as follows:
\begin{align*}
  \lambda_t\sim\Ga(\alpha,1),\qquad
  N_t\mid\lambda_t\sim\Po\big(\orr \lambda_t\big),\qquad
  \zeta_t\mid N_t&\sim\Ga\big(N_t,{\textstyle\frac\beta\rho}\big).
\end{align*}
\end{prop}
\begin{proof}
For $\omega\in\bbR$, 
\begin{alignat*}2
\E e^{i\zeta_t\omega}&\mid& N_t &= (1-i\omega\rho/\beta)^{-N_t}\\
\E e^{i\zeta_t\omega}&\mid& \lambda_t 
  &= \sum_{n=0}^\infty 
   (1-i\omega\rho/\beta)^{-n}
   \big(\orr\lambda_t\big)^n
   \exp\big(-\orr\lambda_t\big)/n!\\
  &&&= \exp\Big(i\frac{(1-\rho)\omega}{\beta-i\rho\omega}\lambda_t\Big)\\
\E e^{i\zeta_t\omega}&&
  &= \Big[1-i\frac{(1-\rho)\omega}{\beta-i\rho\omega}\Big]^{-\alpha}
   = \Big[\frac{\beta-i\omega}{\beta-i\rho\omega}\Big]^{-\alpha}.
\end{alignat*}
\end{proof}
\subsection*{Poisson and Gamma SII Processes}\label{ss:poi-gam}
The chf for a Poisson random variable $X\sim\Po(\nu)$ is
\begin{subequations}
\begin{align}
  \chi_X(\theta)=\E[e^{i\theta X}]
                &=\sum_{k=0}^\infty e^{i\theta k}\set{\frac{\nu^k}{k!}e^{-\nu}}
                 =e^{(e^{i\theta}-1)\nu}\notag
\intertext{so for any $u\in\bbR$ the re-scaled random variable $Y:=uX$ has chf}
  \chi_{uX}(\theta)&=\E[e^{i\theta uX}]=\chi_X(u\theta)\notag\\
                &=e^{(e^{i\theta u}-1)\nu}\notag
\intertext{and a linear combination $Y:=\sum u_j X_j$ of independent
  $X_j\sim\Po(\nu_j)$ has chf}
  \chi_Y(\theta)&=\prod_j\set{e^{(e^{i\theta u_j}-1)\nu_j}}\notag\\
                &=\exp\set{\sum_j {(e^{i\theta u_j}-1)\nu_j}}\label{e:poi-lk-sum}\\
                &=\exp\set{\int_\bbR {(e^{i\theta u}-1)\nu(du)}}\label{e:poi-lk-int}
\end{align}
\end{subequations}
for the discrete measure
\[
        \nu(du) = \sum \nu_j\delta_{u_j}(du)
\]
that assigns mass $\nu_j$ to each point $u_j$, provided the sum in \eqref
{e:poi-lk-sum} converges.  Of course the sum converges if it has only
finitely-many terms, or even if there are infinitely-many with $\sum\nu_j
<\infty$ (because $|e^{i\theta u}-1|\le2$), but that condition isn't
actually necessary.  Since also $|e^{i\theta u}-1|\le|\theta u|$, the
random variable $Y$ will be well-defined and finite provided that
\begin{subequations}
\begin{align}
\sum_j \big(1\wedge|u_j|\big)\nu_j&<\infty\label{e:l1-lk-disc}
\intertext{or, in integral form,}
\int_\bbR \big(1\wedge|u|\big)\nu(du)&<\infty.\label{e:l1-lk}
\end{align}
\end{subequations}
A random variable with chf of form \eqref{e:poi-lk-sum} for a sequence
satisfying \eqref{e:l1-lk-disc} has a ``compound Poisson'' distribution;
one with the more general chf of form \eqref{e:poi-lk-int} for a measure
satisfying \eqref{e:l1-lk} is called Infinitely Divisible, or ID.

\subsubsection*{ID Distributions}\label{sss:ID}

For any $\sigma$-finite measure satisfying \eqref{e:l1-lk} it's easy to
make a stochastic process with stationary independent increments of the
more general form of \eqref {e:poi-lk-int}, beginning with a Poisson random
measure $\cN(du\,ds)$ on $\bbR\times \bbR_+$ with intensity measure
$\E\cN(du\,ds)=\nu(du)\,ds$:
\begin{align}
  X_t &:= \iint_{\bbR\times(0,t]} u \cN(du\,ds).
\end{align}
This is a right-continuous independent-increment nondecreasing process that
begins at $X_0=0$ with jumps
$\Delta_t=[X_t-X_{t-}] \equiv [X_t-\lim_{s\nearrow t}X_s]$ of magnitudes
$\Delta_t\in E$ at rate $\nu(E)$ for any Borel $E\subset\bbR$.  The Poisson
process itself is the special case where $\nu(E)=\nu\bone{1\in E}$, with
jumps of magnitude $\Delta_t=1$ at rate $\nu\in\bbR_+$.

\citet{Khin:Levy:1936} showed that a random variable $Y$ has a chf of the
form \eqref {e:poi-lk-int} if and only\footnote{For nonnegative random
  variables this is true as stated, but a slightly more general form is
  necessary for real-valued ID random variables, with a condition on
  $\nu(du)$ somewhat weaker than \eqref {e:l1-lk} (see \eqref{e:l2-lk})---
  if you get interested, ask me about ``compensation''.  This is needed for
  the $\Ca(\delta,\gamma)$ example and, for $\alpha\ge1$, the
  $\St_0(\alpha,\beta,\gamma,\delta)$ and $\SaS (\alpha,\gamma)$ examples
  below.}
if, for every $n\in\bbN$, one can write $Y=\zeta_1+\cdots\zeta_n$ as the
sum of $n$ iid random variables $\zeta_j$.  This property is called
``Infinite Divisibility'' (abbreviated ID), and the processes we have
constructed whose increments have this property are called ``SII''
processes for their stationary independent increments.  Examples of ID
distributions (or SII processes) and their L\'evy measures
include:\par\medskip

\centerline{\begin{tabular}{lLL}\label{t:id}
Poisson&\Po(\lambda)&\nu(du)=\lambda\delta_1(du)\\
Negative Binomial&\NB(\alpha,p)
  &\nu(du)=\sum_{k\in\bbN}\alpha\frac{q^k}{k}\delta_k(du),\qquad 
  q:=(1-p)\\
Gamma&\Ga(\alpha,\beta)&\nu(du)=\alpha u^{-1}e^{-\beta u}\bone{u>0}\,du\\
$\alpha$-Stable&\St_0(\alpha,\beta,\gamma,\delta)&\nu(du)=
  \frac{\alpha\gamma}\pi\Gamma(\alpha)\sin\pat 
  |u|^{-\alpha-1}(1+\beta\sgn u)\,du\\
Symmetric $\alpha$-Stable&\SaS(\alpha,\gamma)&\nu(du)=
  \frac{\alpha\gamma}\pi\Gamma(\alpha)\sin\pat |u|^{-\alpha-1}\,du\\
Cauchy&\Ca(\delta,\gamma)&\nu(du)=\frac\gamma\pi|u|^{-2}\,du.
\end{tabular}}\par

The defining condition for a random variable $Y$ to be ``Infinitely
Divisible'' (ID) is that for each $n\in\bbN$ there must exist iid random
variables $\{\zeta_j:~1\le j\le n\}$ such that $Y$ and $\sum_{j=1}^n
\zeta_j$ have the same distribution.  This is clearly equivalent to the
condition that every power $\chi^\alpha(\omega)$ of the characteristic
function $\chi(\theta):=\E\exp(i\theta Y)$ must \emph{also} be a
characteristic function (\ie, must be positive-definite) for each inverse
integer $\alpha=1/n$, because we can just take $\{\zeta_j\}$ to be iid with
chf $\chi^{1/n}(\omega)$ and verify that their sum has chf $\chi(\omega)$.
Less obvious but also true is that $Y$ is ID if and only if
$\chi^\alpha(\omega)$ is positive-definite for \emph{all} real $\alpha>0$,
and even less obvious is the \citeauthor {Khin:Levy:1936} theorem that
$\chi$ must take the specific form
\begin{subequations}
\begin{align}
\chi(\theta) &= \exp\set{i\theta\delta - \theta^2\sigma^2/2+
                \int_\bbR \bet{e^{i\theta u}-1}\,\nu(du)}
\label{e:lk-l1}
\intertext{for some $\delta\in\bbR$, $\sigma^2\ge0$, and Borel measure
  $\nu$ satisfying $\nu(\{0\})=0$ and \eqref{e:l1-lk} or, a little more
  generally, the form} \chi(\theta) &= \exp\set{i\theta\delta -
  \theta^2\sigma^2/2+ \int_\bbR \bet{e^{i\theta u}-1-i\theta h(u)
  }\,\nu(du)}
\label{e:lk-l2}
\end{align}
for any bounded function $h$ that satisfies $h(u)=u+O(u^2)$ near
$u\approx0$ (like $\arctan u$ or $u\bone{|u|<1}$ or $u/(1+u^2)$) and a
Borel measure $\nu$ satisfying the weaker restriction
\begin{align}
   \int_\bbR \big(1\wedge u^2\big)\nu(du)&<\infty.\label{e:l2-lk}
\end{align}
\end{subequations}
Some properties of ID distributions beyond the scope of these notes, but
covered in \citep{Steu:vHar:2004} (see also \citep{Bose:DasG:Rubi:2002}),
include:
\begin{thm}[S\&vH, Thm 2.13] 
  Let $\chi(\theta)$ be the chf of an Infinitely Divisible distribution.
  Then $(\forall\theta\in\bbR) \set{\chi(\theta)\ne0}$.  Also, if
  $\chi(\theta)$ is analytic in some open domain $\Omega\subset\bbC$, then
  $(\forall\theta\in\Omega) \set{\chi(\theta)\ne0}$.  Thus, ID chfs do not
  vanish on $\bbR$ or anywhere in $\bbC$ that $\chi(\theta)$ is analytic.
\end{thm}

\begin{thm}[S\&vH, Thm 9.8] 
  Let $X$ be an infinitely divisible random variable that is not normal or
  degenerate.  Then the two-sided tail of $X$ satisfies
  \[ \lim_{x\to\infty}\frac{-\log\P[|X|>x]}{x\log x} = c\] for a number
  $0\le c< \infty$ given by $c^{-1}=\max[\nu(\bbR_+), \nu(\bbR_-)]$.  An ID
  random variable that is not degenerate has a normal distribution if and
  only if the same limit is $c=\infty$.
\end{thm}

\begin{thm}[S\&vH, Prop 2.3] 
No non-degenerate bounded random variable is ID
\end{thm}
This one's easy enough to prove.  If $\|X\|_\infty= B<\infty$ and $X$ has
the same distribution as $\sum_{j=1}^n\zeta_j$ for iid $\{\zeta_j\}$ then
$\|\zeta_j\|_\infty=B/n$ and $\sigma^2:=\V(X)= n\V(\zeta_j) \le
n\E\zeta_j^2 \le B^2/n$, so $\sigma^2=0$ and $X$ must be degenerate.

\subsubsection*{Gamma Variables \& Processes}
The Gamma distribution $X\sim\Ga(\alpha,\beta)$ with mean $\alpha/\beta$
and variance $\alpha/\beta^2$ has chf
\begin{align*}
  \chi_X(\theta)&=\E[e^{i\theta X}]\notag\\
                &=\int_0^\infty e^{i\theta x}
                \set{\frac{\beta^\alpha}{\Gamma(\alpha)} x^{\alpha-1} e^{-\beta
                  x}}\, dx\\ 
                &=\big(1-i\theta/\beta\big)^{-\alpha}\\
                &=\exp\set{-\alpha\log(1-i\theta/\beta)}\\
                &=\exp\set{-\int_0^\infty \big(e^{i\theta u}-1\big)
                  \alpha u^{-1}e^{-\beta u}\,du},
\intertext{exactly of the form of \eqref{e:poi-lk-int} with L\'evy measure}
       \nu(du) &= \alpha u^{-1}e^{-\beta u}\bone{u>0}\,du.
\end{align*}
This measure, while infinite, does satisfy condition \eqref{e:l1-lk}:
\begin{align*}
 \int_\bbR \big(1\wedge|u|\big)\nu(du)
      &= \int_0^1\big(|u|\big)\alpha u^{-1} e^{-\beta u}\,du
          +\int_1^\infty\big(1\big)\alpha u^{-1} e^{-\beta u}\,du\\
      &\le \int_0^\infty\alpha e^{-\beta u}\,du=\alpha/\beta<\infty.
\end{align*}
Thus gamma-distributed random variables are ID
and the SII gamma \emph{process}
\[
  X_t = \iint_{\bbR\times(0,t]} u \cN(du\,ds)
\]
has infinitely-many non-negative ``jumps'' $\Delta_t=X_t-X_{t-}$ in any
time interval $a<t\le b$.  Their sum $\sum\{\Delta_t:~a<t\le b\}=X_b-X_a$
is finite, however, with probability distribution
\[ X_b-X_a \sim \Ga\big(\alpha(b-a),~\beta\big). \]
The gamma \emph{random measure} $\cG(dx\,dy)\sim\Ga(\alpha\,dx\,dy,~\beta)$
of \Sec{ss:rm} has a similar construction,
\[
  \cG(A) = \iiint_{\bbR\times A} u \cN(du\,dx\,dy)
\]
as the sum of the heights $u_j$ of a Poisson cloud of points
$(u_j,x_j,y_j)$ for which $(x_j,y_j)\in A$.  \citet{Wolp:Icks:1998} show
how to simulate such random measures very efficiently, drawing the jumps
$\{u_j\}$ in monotone decreasing order.
\vfill\newpage
\ifBib
\bibliography{arp}

\begin{thebibliography}{11}
\providecommand{\natexlab}[1]{#1}
\providecommand{\url}[1]{\texttt{#1}}
\expandafter\ifx\csname urlstyle\endcsname\relax
  \providecommand{\doi}[1]{doi: #1}\else
  \providecommand{\doi}{doi: \begingroup \urlstyle{rm}\Url}\fi

\bibitem[Bose et~al.(2002)Bose, DasGupta, and Rubin]{Bose:DasG:Rubi:2002}
Arup Bose, Anirban DasGupta, and Herman Rubin.
\newblock A contemporary review and bibliography of infinitely divisible
  distributions and processes.
\newblock \emph{Sankhy\={a}, Ser. {\rm A}}, 64\penalty0 (3):\penalty0 763--819,
  2002.
\newblock \doi{10.2307/25051430}.
\newblock Special issue in memory of D. Basu.

\bibitem[Cox et~al.(1985)Cox, Ingersoll, and Ross]{Cox:Inge:Ross:1985}
John~C. Cox, Jonathan~E. Ingersoll, Jr., and Stephen~A. Ross.
\newblock A theory of the term structure of interest rates.
\newblock \emph{Econometrica}, 53\penalty0 (2):\penalty0 385--408, 1985.
\newblock \doi{10.2307/1911242}.

\bibitem[Feller(1951)]{Fell:1951}
William Feller.
\newblock Two singular diffusion problems.
\newblock \emph{Annals of Mathematics}, 54\penalty0 (1):\penalty0 173--182,
  1951.
\newblock \doi{10.2307/1969318}.

\bibitem[Hjort(1980)]{Hjor:1980}
Nils~Lid Hjort.
\newblock Nonparametric {B}ayes estimators based on beta processes in models
  for life history data.
\newblock \emph{Ann. Stat.}, 18\penalty0 (3):\penalty0 1259--1294, 1980.
\newblock \doi{10.1214/aos/1176347749}.

\bibitem[Khinchine and L{\'e}vy(1936)]{Khin:Levy:1936}
Alexander~Ya. Khinchine and Paul L{\'e}vy.
\newblock Sur les lois stables.
\newblock \emph{Comptes rendus hebdomadaires des seances de l'Acad{\'e}mie des
  sciences. Acad{\'e}mie des science (France). Serie A. Paris}, 202:\penalty0
  374--376, 1936.

\bibitem[Steutel and van Harn(2004)]{Steu:vHar:2004}
Frederik~W. Steutel and Klaas van Harn.
\newblock \emph{Infinite Divisibility of Probability Distributions on the Real
  Line}, volume 259 of \emph{Monographs and Textbooks in Pure and Applied
  Mathematics}.
\newblock Marcel Dekker, New York, NY, 2004.
\newblock ISBN 0-8247-0724-9.
\newblock \doi{10.1201/9780203014127}.

\bibitem[Thibaux and Jordan(2007)]{Thib:Jord:2007}
Romain Thibaux and Michael~I. Jordan.
\newblock Hierarchical beta processes and the {I}ndian buffet process.
\newblock In Marina Meila and Xiaotong Shen, editors, \emph{Proceedings of the
  Eleventh International Conference on Artificial Intelligence and Statistics
  (AISTATS 2007), March 21--24, San Juan, Puerto Rico}, volume~2, pages
  564--571. PMLR, 2007.
\newblock URL \url{http://proceedings.mlr.press/v2/thibaux07a.html}.

\bibitem[Walker(2000)]{Walk:2000}
Stephen~G. Walker.
\newblock A note on the innovation distribution of a gamma distributed
  autoregressive process.
\newblock \emph{Scand. J.~Stat.}, 27\penalty0 (4):\penalty0 575--576, 2000.
\newblock \doi{10.1111/1467-9469.00208}.

\bibitem[Wang(2013)]{Wang:2013}
Jianyu Wang.
\newblock \emph{Bayesian Modeling and Adaptive Monte Carlo with Geophysics
  Applications}.
\newblock PhD thesis, Duke University Statistical Science Department, 2013.

\bibitem[Wolpert and Brown(2021)]{Wolp:Brow:2017}
Robert~L. Wolpert and Lawrence~D. Brown.
\newblock Markov infinitely-divisible stationary time-reversible integer-valued
  processess.
\newblock On arXiv.

\bibitem[Wolpert and Ickstadt(1998)]{Wolp:Icks:1998}
Robert~L. Wolpert and Katja Ickstadt.
\newblock Poisson/gamma random field models for spatial statistics.
\newblock \emph{Biometrika}, 85\penalty0 (2):\penalty0 251--267, 1998.
\newblock \doi{10.1093/biomet/85.2.251}.

\end{thebibliography}
\else
\newcommand{\noopsort}[1]{}

\fi 
\end{document}
\vfill\hfill{\tiny Last edited: \today}